\documentclass[amssymb,amsfonts,reqno]{amsart}
\usepackage{amssymb,amsmath,amsfonts}
\usepackage{latexsym}
\usepackage{graphicx}

\def\be#1{\begin{equation}\label{#1}}
\def\bas{\begin{align*}}
\def\eas{\end{align*}}
\def\bi{\begin{itemize}}
\def\ei{\end{itemize}}

\usepackage{pdfpages}
\usepackage[framemethod=TikZ]{mdframed}

\mdfdefinestyle{JimFrame}{%
    linecolor=red,
    outerlinewidth=2pt,
    roundcorner=20pt,
    innertopmargin=\baselineskip,
    innerbottommargin=\baselineskip,
    innerrightmargin=25pt,
    innerleftmargin=25pt,
    backgroundcolor=white!50!white}

\theoremstyle{plain}
   \newtheorem{theorem}[subsection]{Theorem}

   \newtheorem{proposition}[subsection]{Proposition}
   
   \newtheorem{lemma}[subsection]{Lemma}

\usepackage{verbatim}
\usepackage{color}

\begin{document}

\author{Roberto Bramati}
\address{Universit\'e de Lorraine, CNRS, IECL, F-57000 Metz, France} 
\email{roberto.bramati@univ-lorraine.fr}

\author{Paolo Ciatti}
\address{Dipartimento di Ingegneria Civile, Edile e Ambientale, via Marzolo 9, 35131 Padova,
Italia}
\email{paolo.ciatti@unipd.it}

\author{John Green}
\address{Maxwell Institute of Mathematical Sciences and the School of Mathematics, University of
Edinburgh, JCMB, The King's Buildings, Peter Guthrie Tait Road, Edinburgh, EH9 3FD, Scotland}
\email{j.d.green@sms.ed.ac.uk}

\author{James Wright}
\address{Maxwell Institute of Mathematical Sciences and the School of Mathematics, University of
Edinburgh, JCMB, The King's Buildings, Peter Guthrie Tait Road, Edinburgh, EH9 3FD, Scotland}
\email{j.r.wright@ed.ac.uk}


\thanks{The first and the second author were partially supported by GNAMPA (Project 2020 “Alla frontiera tra l'analisi complessa in pi\`u variabili e l'analisi armonica”).}

\title[Oscillating multipliers on Heisenberg type groups]{Oscillating spectral multipliers on groups of Heisenberg type}

\begin{abstract}
We establish endpoint estimates for a class of oscillating spectral multipliers
on Lie groups of Heisenberg type. The analysis follows an earlier argument due to
the second and fourth author \cite{CW} but requires the detailed analysis of the wave equation
on these groups due to  M\"uller and Seeger \cite{MSeeger}. 

We highlight and develop the connection between sharp bounds for oscillating multipliers
and the problem of determining the minimal amount of smoothness required for
Mihlin-H\"ormander multipliers, a problem that was solved for groups of Heisenberg type but remains open for other groups.

\end{abstract}

\maketitle

\section{Introduction}
 
Let $G$ be a connected  Lie group and let $X_1, X_2, \ldots, X_n$ be left invariant vector fields
on $G$ which satisfy H\"ormander's condition; that is, they generate, together with the iterative commutators,
the tangent space of $G$ at every point (a special case is a Lie group of Heisenberg type; see
the next section for precise definitions). The sublaplacian 
${\mathcal L} = - \sum_{j=1}^n X_j^2$ is then a nonnegative, second order
hypoelliptic operator which is essentially
self-adjoint on $L^2(\mu)$ where $\mu$ is the right Haar measure on $G$. Hence $\sqrt{\mathcal L}$ admits a spectral
resolution $\{E_{\lambda}\}$; if $m$ is a bounded, Borel measurable function on $[0,\infty)$, then
the spectral theorem implies that
$$
m({\sqrt{\mathcal L}}) \ = \ \int_0^{\infty}m(\lambda) \, dE_{\lambda}
$$
is a bounded operator on $L^2(G)$. 

In this paper we will consider a general framework of spectral multipliers which contains oscillating examples of the form 
\begin{equation}\label{osc-multipliers}
m_{\theta, \beta}(\lambda)  \ = \ \frac{e^{i \lambda^{\theta}}}{\lambda^{\theta \beta/2}} \, \chi(\lambda)
\end{equation}
for any $\theta, \beta \ge 0$. Here $\chi \in C^{\infty}({\mathbb R})$ satisfies
$\chi(\lambda) \equiv 0$
for $\lambda\le 1$ and $\chi(\lambda) \equiv 1$ when $\lambda$ is large.
A consequence of our analysis is the following result.

\begin{theorem}\label{main-consequence}
Let $G$ be a Lie group of Heisenberg type and $1<p$. For all $\theta \ge 0$ with $\theta \not= 1$,  
$m_{\theta, \beta}(\sqrt{\mathcal L})$ is  bounded on $L^p(G)$ if and only if $\beta/2 \ \ge \ d \,  | 1/p - 1/2|$. Furthermore $m_{1,\beta}(\sqrt{\mathcal L})$ is bounded on $L^p(G)$ if and only if
$\beta/2 \ \ge \ (d-1) \, |1/p - 1/2|$. 
\end{theorem}

The second part is a statement about the wave operator on groups of Heisenberg type and is not new. The sufficiency
is established by M\"uller and Seeger \cite{MSeeger} (see earlier work M\"uller and Stein \cite{MS1}) and the necessity
follows from recent work in \cite{GMM} which is valid on stratified groups of arbitrary step. 
The sufficiency of the first part is only new
at the endpoint $\beta/2 = d |1/p - 1/2|$.  The open range $\beta/2 > d |1/p - 1/2|$ follows from work on
a closely related problem concerning Mihlin-H\"ormander spectral multipliers.

The analysis in \cite{GMM} and \cite{gafa} shows
that for the  oscillating multipliers 
$m_{2,\beta}(\lambda) = e^{i \lambda^{2}}\lambda^{-\beta} \, \chi(\lambda)$,
the operator $m_{2,\beta}(\sqrt{\mathcal L})$ is unbounded on $L^p(G)$ when $\beta/2 < d |1/p - 1/2|$;
this being the case of the Schr\"odinger group and is valid for
sublaplacians ${\mathcal L}$ on any stratified Lie group $G$. 
This gives the necessity of the first
statement in Theorem \ref{main-consequence}.

Theorem \ref{main-consequence} is a consequence of a more general theorem. See Section \ref{framework}.

An interesting, ongoing problem is to determine the minimal amount
of smoothness for spectral multipliers $m$ which guarantees that
$m(\sqrt{\mathcal L})$ is bounded on all $L^p, 1<p<\infty$ (and weak-type $(1,1)$, bounded on $H^1$, etc...).
This problem has been studied extensively in the context of Mihlin-H\"ormander multipliers where a scale-invariant
Sobolev condition of the form
\begin{equation}\label{hormander}
\|m\|_{L^q_{s, sloc}} \ := \ \sup_{t\ge 0} \| m(t \cdot) \chi \|_{L^q_s({\mathbb R}_{+})} \ < \ \infty
\end{equation}
is imposed for an appropriate $q\in [1,\infty]$ and $s \in [0,\infty)$; here\footnote{We will use $\chi$ to denote smooth
cut-off functions in different situations (as in \eqref{osc-multipliers}). The context should be clear.}
$\chi \in C^{\infty}_c(0,\infty)$ is any nontrivial smooth cut-off function and $L^q_s$ is the $L^q$ Sobolev
space of order $s$. Following Martini and M\"uller \cite{gafa} (see also \cite{GMM}), we denote the Mihlin-H\"ormander
threshold ${\mathfrak s}({\mathcal L})$ as the infimum in $s \in [0,\infty)$ where 
\begin{equation}\label{zL}
\forall p \in (1,\infty), \ \exists \, C = C_{p,s} <\infty \ \ {\rm such \ that} \ \ \|m(\sqrt{\mathcal L})\|_{p\to p} \le C \|m\|_{L^2_{s, sloc}}
\end{equation}
holds for all bounded, Borel measurable $m$. 

When $G = {\mathbb R}^d$ and ${\mathcal L} = - \Delta$ is the usual
Laplacian, then it is well known that ${\mathfrak s}(-\Delta) = d/2$. A fundamental result of Christ \cite{C}
and Mauceri-Meda \cite{MM} shows that on any stratified Lie group $G$,  ${\mathfrak s}({\mathcal L}) \le Q/2$ where $Q$ is
the {\it homogeneous dimension} of $G$.  Shortly afterwards, Hebisch \cite{hebisch}
and  M\"uller and Stein \cite{MS}  observed that when $G$ is a group of Heisenberg type, then ${\mathfrak s}({\mathcal L}) \le d/2$
where $d$ is the topological dimension of $G$. \footnote{For any stratified Lie group $G$, the topological dimension
is always strictly smaller than the homogeneous dimension, except when $G$ is Euclidean.} A number of results
in this direction have been obtained since then, and we now know that $d/2 \le {\mathfrak s}({\mathcal L}) < Q/2$ for
any 2-step stratified group $G$ \cite{gafa} and that equality ${\mathfrak s}({\mathcal L}) = d/2$ holds in a number
of cases; see \cite{free}, \cite{plms} and \cite{alessio-reiter}.

Let $G$ be a Lie group of Heisenberg type (in particular a $2$-step stratified group) so that ${\mathfrak s}({\mathcal L}) = d/2$. 
Since $\|m\|_{L^2_{s,sloc}}$ is smaller than $\|m\|_{L^{\infty}_{s, sloc}}$ and 
$\|m\|_{L^{\infty}_{0,sloc}} \sim \, \|m\|_{L^{\infty}}$,
we can interpolate the estimates in \eqref{zL} with trivial
$L^2$ bounds to conclude that 
\begin{equation}\label{Linfty}
\|m(\sqrt{\mathcal L})\|_{p\to p} \ \le \ C_{p, s} \, \|m\|_{L^{\infty}_{s, sloc}}
\end{equation}
holds for all $p \in (1,\infty)$ and $s > d | 1/p - 1/2|$. In particular, since the oscillating multiplier
$m_{\theta, \beta}(\lambda)$ in \eqref{osc-multipliers} satisfies $\|m_{\theta, \beta}\|_{L^{\infty}_{\beta/2, sloc}} < \infty$,
we see that $m_{\theta, \beta}(\sqrt{\mathcal L})$ is bounded on $L^p(G)$ if $d |1/p - 1/2| < \beta/2$. Hence
it is the endpoint $L^p$ bound when $\beta/2 = d |1/p - 1/2|$ which interests us in this paper.

The estimates \eqref{Linfty} motivate the interest in another threshold exponent ${\mathfrak s}_{-}({\mathcal L})$, also 
introduced in \cite{gafa}, defined as the infimum in $s \in [0,\infty)$ where 
\begin{equation}\label{zL-}
\forall p \in (1,\infty), \ \exists \, C = C_{p,s} <\infty \ \ {\rm such \ that} \ \ \|m(\sqrt{\mathcal L})\|_{p\to p} \le C \|m\|_{L^{\infty}_{s, sloc}}
\end{equation}
holds for all bounded, Borel measurable $m$. By interpolation 
${\mathfrak s}_{-}({\mathcal L}) \le {\mathfrak s}({\mathcal L})$ and so the exponent 
${\mathfrak s}_{-}({\mathcal L})$ can be used
to provide lower bounds for ${\mathfrak s}({\mathcal L})$. 

In \cite{GMM}, the lower bound $d/2 \le {\mathfrak s}_{-}({\mathcal L})$ was established in great generality, including 
sublaplacians ${\mathcal L}$ on
any stratified Lie group of arbitrary step. In fact in \cite{GMM}, it was shown that if \eqref{Linfty} holds
for some $1\le p$ and $s\ge 0$ and all bounded Borel measurable $m$, then necessarily $s \ge  d |1/p - 1/2|$.
This implies that $d/2 \le {\mathfrak s}_{-}({\mathcal L}) \ (\le {\mathfrak s}({\mathcal L}))$ in this case. 
The same conclusions were obtained in \cite{gafa} in less generality; for sublaplacians on
any step 2 stratified group, with less robust methods. 


Yet another threshold exponent ${\mathfrak s}_{+}({\mathcal L})$ was introduced
in \cite{gafa} which will be useful for us. It is defined as the infimum in $s \in [0,\infty)$ where 
\begin{equation}\label{zL+}\, 
\forall \ {\rm compact}\,  K \subset {\mathbb R}, \ \exists  C_{K,s} <\infty \  {\rm such \ that} \ 
\|F(\sqrt{\mathcal L})\|_{1\to 1} \le C_{K,s} \|F\|_{L^{2}_{s}}
\end{equation}
holds for all Borel measurable $F$ with ${\rm supp}(F) \subset K$. The estimate \eqref{zL+} is the key estimate
behind establishing the endpoint bound $\|m(\sqrt{\mathcal L})\|_{L^1 \to L^{1,\infty}} \le C \|m\|_{L^2_{s, sloc}}$ via
Calder\'on-Zygmund theory. Hence ${\mathfrak s}({\mathcal L}) \le {\mathfrak s}_{+}({\mathcal L})$; see \cite{alessio-fourier}.  
In \cite{gafa}, the inequalities 
$d/2 \le {\mathfrak s}_{-}({\mathcal L}) \le {\mathfrak s}({\mathcal L}) \le {\mathfrak s}_{+}({\mathcal L}) < Q/2$
were established for any 2-step stratified Lie group although in many cases, including the case of groups
of Heisenberg type, we have 
$d/2 = {\mathfrak s}_{-}({\mathcal L}) = {\mathfrak s}({\mathcal L}) = {\mathfrak s}_{+}({\mathcal L})$. 

To see the relevance of \eqref{zL+}, consider the natural decomposition 
\begin{equation}\label{basic-decomp}
m(\lambda) \ = \ \sum_{j\in {\mathbb Z}} m(\lambda) \phi(2^{-j} \lambda) \ =: \ \sum_{j\in {\mathbb Z}} m_j(\lambda)
\end{equation}
of a general spectral multiplier $m$ where $\phi \in C^{\infty}_c (0, \infty)$ satisfies 
$\sum_{j \in {\mathbb Z}} \phi(2^{-j} \lambda) \equiv 1$.  Setting $m^j (\lambda) := m(2^j \lambda) \phi(\lambda)$,
we see that for our oscillating multipliers $m_{\theta, \beta}$, we have the uniform bound
\begin{equation}\label{L2-s-beta}
\|m^j_{\theta, \beta} \|_{L^2_s} \ \le \ C \, 2^{|j|\theta (s - \beta/2)}, \ \ \forall s \ge 0.
\end{equation}
Therefore if 
$\beta > 2 {\mathfrak s}_{+}({\mathcal L})$, we can find an $s$ with ${\mathfrak s}_{+}({\mathcal L}) < s < \beta/2$ and hence
\begin{align*}
\|m_{\theta, \beta}(\sqrt{\mathcal L})\|_{1 \to 1} &\le \sum_{j\in {\mathbb Z}} \|m^j_{\theta, \beta}(\sqrt{\mathcal L})\|_{1 \to 1}
\\&\le C_s \sum_{j\in {\mathbb Z}} \|m^j\|_{L^2_s} \le C_s \sum_{j\in {\mathbb Z}} 2^{-|j|\theta (\beta/2 - s)},
\end{align*}
showing that $m_{\theta, \beta}(\sqrt{\mathcal L})$ is bounded on $L^1(G)$ when $\beta > 2 {\mathfrak s}_{+}({\mathcal L})$. 
Now when $\beta < 2 s_{+}$ where $s_{+} = {\mathfrak s}_{+}({\mathcal L})$, we can embed $m_{\theta, \beta}$  
into the analytic family of multipliers
$m_z (\lambda) = \lambda^{\theta/2(\beta - (2 s_{+}+\delta) z)} m_{\theta, \beta}(\lambda)$  and use 
Stein's analytic  interpolation theorem
to conclude the following.

\begin{lemma}\label{open-range} When $\beta < 2 {\mathfrak s}_{+}({\mathcal L})$, then 
 $m_{\theta, \beta}(\sqrt{\mathcal L})$ is bounded on 
$L^p(G)$ whenever $2 {\mathfrak s}_{+}({\mathcal L}) |1/p - 1/2| < \beta/2$.
\end{lemma}
In particular when ${\mathfrak s}_{+}({\mathcal L}) = d/2$, this gives us an alternative proof of the $L^p$ boundedness 
for $m_{\theta, \beta}(\sqrt{\mathcal L})$ in 
the open range $d |1/p - 1/2| < \beta/2$. 

A natural way to establish the endpoint $L^p$ bound
when $d |1/p - 1/2| = \beta/2$ is to embed the oscillating multiplier $m_{\theta, \beta}$ with $\beta < d$  into the analytic family
${\mathfrak m}_z(\lambda) = \lambda^{\theta/2 (\beta - d z)} m_{\theta, \beta}$ of multipliers so that
${\mathfrak m}_{\beta/d} = m_{\theta, \beta}$ and
prove some Hardy space estimate $H^1 \to L^1$  for ${\mathfrak m}_{1 + iy}$ 
with polynomial bounds in $|y|$. Since ${\mathfrak m}_{iy} \in L^{\infty}$, uniformly in $y$, the multiplier
operator is uniformly bounded on $L^2$ and so 
Stein's
analytic interpolation theorem can then be invoked to show ${\mathfrak m}_{\beta/d} = m_{\theta, \beta}$
is bounded on $L^p$ for $d |1/p - 1/2| = \beta/2$. This will be the procedure we will follow.

\subsection*{Notation}
We use the notation $A \lesssim B$ between two positive quantities $A$ and $B$ to denote $A \le CB$ for some
constant $C$. We sometimes use the notation $A \lesssim_k B$ to emphasize that the
implicit constant depends on the parameter $k$. We sometimes use $A = O(B)$ to denote the
inequality $A \lesssim B$.
Furthermore, we use $A \ll B$ to denote $A \le \delta B$ for a sufficiently small constant $\delta>0$
whose smallness will depend on the context.

\subsection*{Outline of paper}
In the next section we review the definition of Lie groups of Heisenberg type and recall
the key results from \cite{MSeeger} where M\"uller and Seeger give a detailed 
analysis of the wave equation in this setting, including the introduction of
a local, isotropic Hardy space $h^1_{iso}(G)$ which is compatible with the underlying group structure.
In Section \ref{framework}
we develop a general framework of spectral multipliers which include both Mihlin-H\"ormander multipliers
as well as the oscillating examples \eqref{osc-multipliers}.  We formulate the main estimate in Theorem \ref{main-consequence}
in this more general framework. In Section \ref{main-H1-estimate} we give the
proof of this main estimate on $h^1_{iso}(G)$, up to the final step which requires a fine decomposition  of
the wave operator in \cite{MSeeger}. We describe this decomposition  in Section \ref{decomposition}. In Section \ref{final-step},
we provide the final step in the the proof. 


\section{Groups of Heisenberg type: the work of M\"uller-Seeger \cite{MSeeger}}
Let $G$ be a connected Lie group and let ${\mathfrak g}$ denote the associated Lie algebra. 
We say that $G$ is a Lie group of Heisenberg type if its Lie algebra ${\mathfrak g}$ is a Lie algebra of Heisenberg type. A Lie algebra of Heisenberg type is a 2-step stratified Lie algebra, i.e., ${\mathfrak g}= {\mathfrak g}_1 \oplus {\mathfrak g}_2$, where ${\mathfrak g}_1$, ${\mathfrak g}_2$ are subspaces of dimensions $d_1,d_2$ satisfying $[ {\mathfrak g}, {\mathfrak g}] \subseteq {\mathfrak g}_2\subseteq {\mathfrak z}({\mathfrak g})$, where $ {\mathfrak z}({\mathfrak g})$ is the centre of ${\mathfrak g}$. Moreover the definition requires that, endowing $\mathfrak g$ with an inner product $\langle \cdot, \cdot \rangle$ such that ${\mathfrak g}_1$ and 
${\mathfrak g}_2$ are orthogonal subspaces, the unique skew-symmetric endomorphisms $J_{\mu}$ on ${\mathfrak g}_1$, with $\mu \in {\mathfrak g}_2^{*}\setminus\{0\}$, defined by 
$$ 
 \langle J_{\mu} (V), W \rangle \ = \ \mu([V,W])  \ \ {\rm for \ all \ pairs} \ \ V, W \in {\mathfrak g}_1
$$
satisfy $J_{\mu}^2 = - |\mu|^2 I$. In particular this implies that ${\rm dim}\, {\mathfrak g}_1  = d_1$ is even. 
 
We fix an orthonormal basis $X_1, \ldots, X_{d_1}$ of ${\mathfrak g}_1$ 
and an orthonormal basis $U_1, \ldots, U_{d_2}$ of ${\mathfrak g}_2$. We identify the dual spaces
${\mathfrak g}_1^{*}$ and ${\mathfrak g}_2^{*}$ with ${\mathfrak g}_1$ and ${\mathfrak g}_2$
via the inner product.

From now on, $G$ will denote a Lie group of Heisenberg type (and ${\mathfrak g}$ will denote
a Lie algebra of Heisenberg type).

The operators ${\mathcal L}, -i U_1, \ldots, -i U_{d_2}$ form a set of positive strongly commuting self-adjoint
operators and admit a joint spectral resolution. However we will only need to consider operators of the form
$\phi({\mathcal L}, |U|)$ where $U := (-i U_1, \ldots, -i U_{d_2})$. 

We will identify $G$ with its Lie algebra ${\mathfrak g} = {\mathfrak g}_1 \oplus {\mathfrak g}_2 \simeq 
{\mathbb R}^{d_1} \times {\mathbb R}^{d_2}$ via the exponential map
and we write points in $G$ as $(x,u)$ where $x \in {\mathfrak g}_1 \simeq {\mathbb R}^{d_1}$
and $u \in {\mathfrak g}_2 \simeq {\mathbb R}^{d_2}$. The topological dimension of $G$
is $d = d_1 + d_2$ and the homogeneous dimension is $Q = d_1 + 2 d_2$. We can write the 
group law on $G$ as
$$
(x,u) \cdot (x', u') \ = \ (x + x', u + u' + \frac{1}{2} \langle J x, x' \rangle )
$$
where $x\in {\mathfrak g}_1$ and $u \in {\mathfrak g}_2$ and $\langle J x, x' \rangle $ denotes a vector
in ${\mathfrak g}_2$ with components $\langle J_{U_i} x, x' \rangle$.

Consider the positive sublaplacian
$$
{\mathcal L} \ = \ - (X_1^2 \ + \ \cdots \ + \ X_{d_1}^2)
$$  
with spectral resolution $\sqrt{\mathcal L} = \int_0^{\infty} \lambda \,  dE_{\lambda}$. 
Then 
\begin{equation}\label{spectral-resolution}
m(\sqrt{\mathcal L}) \ = \ \int_0^{\infty} m(\lambda) \, d E_{\lambda}
\end{equation}
defines a spectral multiplier operator
which is  bounded on $L^2(G)$ precisely when
$m$ is a bounded, Borel measurable function on ${\mathbb R}_{+} = [0,\infty)$. 

Abusing notation, we will also denote by $m(\sqrt{{\mathcal L}})$
the convolution kernel of the operator $m(\sqrt{{\mathcal L}})$. 

The main result of M\"uller and Seeger in \cite{MSeeger} states that the wave operator
$m_{1, \beta}(\sqrt{\mathcal L}) = e^{i \sqrt{\mathcal L}} (1 + \sqrt{\mathcal L})^{-\beta/2}$ on a group $G$
of Heisenberg type is bounded on $L^p(G)$ when $\beta/2 = (d-1) |1/p - 1/2|$. Recall that we
included this in the statement of Theorem \ref{main-consequence}. 

Hence solutions
$$
u(\cdot, \tau) \ = \ \cos(\tau \sqrt{\mathcal L}) \, f \ + \ \frac{\sin(\tau \sqrt{\mathcal L})}{\sqrt{\mathcal L}} \, g
$$
of the Cauchy problem
$$
(\partial^2_{\tau} + {\mathcal  L}) u \ = \ 0, \ \ u |_{\tau =0} = f, \ \ \partial_{\tau} u |_{\tau = 0} = g 
$$
satisfy the Sobolev inequality 
\begin{equation}\label{sobolev}
\|u(\cdot, \tau)\|_p \ \le \ C \, \bigl[ \|(1 + \tau^2 {\mathcal  L})^{\gamma/2} f \|_p \ 
+ \ \| \tau (1 + \tau^2 {\mathcal L})^{\gamma/2 - 1} g \|_p \bigr]
\end{equation}
where $\gamma = (d-1) |1/p - 1/2|$. 

Their proof involves a detailed analysis of the singularities of the wave kernel on groups of Heisenberg type
and an appropriate corresponding Littlewood-Paley type decomposition of the wave operator. Their analysis  
also gives sharp $L^1$ estimates for wave operators whose symbols are supported in
a dyadic interval. 

\begin{theorem}\label{MS-basic} \cite{MSeeger} Let $\chi \in C^{\infty}_c$ be supported in $(1/2, 2)$.
Then the $L^1$ operator norm of $\chi(\lambda^{-1} \sqrt{\mathcal L}) e^{i \sqrt{\mathcal L}}$ has the following bound:
\begin{equation}\label{L1}
\|\chi(\lambda^{-1} \sqrt{\mathcal L}) e^{i \sqrt{\mathcal L}}\|_{L^1(G) \to L^1(G)} \ \le \  C  \,  (1+|\lambda|)^{(d-1)/2}.
\end{equation}
\end{theorem} 

Such $L^1$ estimates immediately show that ${\mathfrak s}_{+}({\mathcal L}) \le d/2$ on groups of Heisenberg type.

\begin{proposition}\label{L1-key} For all $F$, compactly supported in ${\mathbb R}^{+}$
and for all $s > d/2$, we can find a constant $C = C_s$ such that
\begin{equation}\label{key-estimate}
\int_G |F(\sqrt{\mathcal L})(x)| \, dx \ \le \ C_s \, \|F\|_{L^{2}_s}
\end{equation}
holds. Here we are employing our convention that $F(\sqrt{\mathcal L})(x)$ also denotes the
convolution kernel of the operator  $F(\sqrt{\mathcal L})$.
\end{proposition}

\begin{proof} By the Fourier inversion formula, we write
$$
F(t) \ = \ \int {\widehat F}(\tau) e^{i t \tau} \, d\tau.
$$
Let $\chi \in C^{\infty}_c({\mathbb R})$ such that $F(t) = \chi(t) F(t)$ and hence 
\begin{equation}\label{fourier-rep}
F(\sqrt{\mathcal L}) \ = \ \int {\widehat F}(\tau) \chi(\sqrt{\mathcal L}) e^{i \tau \sqrt{\mathcal L}} \, d\tau.
\end{equation}
By Theorem \ref{MS-basic}, we see that the operator $\chi(\sqrt{\mathcal L}) e^{i\tau \sqrt{\mathcal L}}$
is bounded on $L^1(G)$ with operator norm $O((1+|\tau|)^{(d-1)/2})$. Hence by Cauchy-Schwarz,
$$
\|F(\sqrt{\mathcal L})\|_{L^1(G)} \ \le \  \int |{\widehat F}(\tau)| (1+|\tau|)^{(d-1)/2} \, d\tau \ \lesssim_s \
\|F\|_{L^2_s}
$$
for any $s > d/2$, establishing \eqref{key-estimate}.
\end{proof}

As we already mentioned, the estimate \eqref{key-estimate} implies (via Calder\'on-Zygmund techniques;
see \cite{alessio-fourier}) that
\begin{equation}\label{L1-consequence}
\|m(\sqrt{\mathcal L}) \|_{L^1 \to L^{1,\infty}} \ \ {\rm and} \ \  \|m(\sqrt{\mathcal L}) \|_{L^p \to L^p} \ \le \ C \,
\|m\|_{L^2_{s, sloc}}
\end{equation}
for any $s>d/2$ and all $1<p <\infty$.

In \cite{MSeeger}, M\"uller and Seeger give an alternative proof of \eqref{L1-consequence} which directly uses
the $L^1$ operator norm bound of $\chi(\lambda^{-1} \sqrt{\mathcal L}) e^{i \sqrt{\mathcal L}}$ in
\eqref{L1}, together with a pointwise bound on the convolution kernel of $\chi(\lambda^{-1} \sqrt{\mathcal L}) e^{i \sqrt{\mathcal L}}$ 
which is easily derived from 
the fundamental finite propagation speed property of solutions of the wave
equation in this context (see \cite{melrose}).

 In fact the real part of the convolution kernel of
$\chi(\lambda^{-1} \sqrt{\mathcal L}) e^{ i \sqrt{\mathcal L}}$ can be written as $\varphi_{\lambda} * {\mathcal P}$
where $\varphi_{\lambda}$ and 
${\mathcal P}$ are the convolution kernels of $\chi(\lambda^{-1} \sqrt{\mathcal L})$
and $\cos(\sqrt{\mathcal L})$, respectively. Here $\varphi_{\lambda}(x,u) = \lambda^Q \varphi(\delta_{\lambda}(x,u))$ is a dilate of
a Schwartz function $\varphi$ and
the finite speed property implies that ${\mathcal P}$ is a compactly
supported distribution of finite order. Hence there exists an $M\ge 1$ such that the bound
$|\varphi_{\lambda} * P (x,u)| \lesssim_N \,  \lambda^M \|\delta_{\lambda}(x,u)\|_E^{-N}$
holds for large $\lambda$ and any $N\ge 1$  (here $\| \cdot\|_E$ denotes the Euclidean norm on $G$). A similar bound
holds for the imaginary part (see \cite{MSeeger},  Proposition 8.8 for further details). 

We record this
bound in the following proposition.

\begin{proposition} For $\lambda \ge 1$ and $\chi \in {\mathcal S}({\mathbb R})$, the convolution kernel of
the operator $\chi(\lambda^{-1} \sqrt{\mathcal L}) e^{i \sqrt{\mathcal L}}$ has the following bound; there exists an $M\ge 1$ such that
\begin{equation}\label{finite-speed}
\bigl| \chi(\lambda^{-1} \sqrt{\mathcal L}) e^{i \sqrt{\mathcal L}} (x,u) \bigr| \ \le \ C_N  \lambda^M (\lambda |x| + \lambda^2 |u|)^{-N}
\end{equation}
holds for any $N\ge 1$. Here $C_N$ depends only on $N$ and a suitable Schwartz norm of $\chi$.
\end{proposition}

M\"uller and Seeger
use \eqref{L1} and \eqref{finite-speed} to give a short
proof of \eqref{L1-consequence}. Their argument also
shows that
\begin{equation}\label{MS-H1}
\|m(\sqrt{\mathcal L}) \|_{H^1(G) \to L^1(G)} \  \ \le \ C \,
\|m\|_{L^2_{s, sloc}}
\end{equation}
for any $s>d/2$. Here $H^1(G)$ denotes the Hardy space on $G$ defined with respect to the 
nonisotropic automorphic dilations
\begin{equation}\label{auto-dilations}
\delta_r (x, u) \ := \ (rx, r^2 u), \ \ \ r >0,
\end{equation} 
together with the Kor\'anyi balls
\begin{equation}\label{Koranyi}
B_r (x,u) \ := \ \bigl\{ (y,v) \in G:  \ \| (y,v)^{-1} \cdot (x,u) \|_K < r \bigr\}
\end{equation}
where $\| (x,u)\|_K := (|x|^4 + |4u|^2 )^{1/4}$ defines the Kor\'anyi norm on $G$. 

As outlined at the end of the Introduction, to prove the endpoint $L^p$ bound for the
oscillating spectral multipliers $m_{\theta, \beta}$ it suffices to embed $m_{\theta, \beta}$
into an analytic family ${\mathfrak m}_z$ and establish an appropriate
Hardy space bound $H^1 \to L^1$ for ${\mathfrak m}_{1+iy}$ corresponding to $m_{\theta, d}$. It is natural to try to do
this with the Hardy space $H^1(G)$ described above, defined with the nonisotropic automorphic
dilations \eqref{auto-dilations}. However we are unable to do this; the geometry of the Kor\'anyi balls
\eqref{Koranyi} does not seem appropriate for our problem. Instead we need a Hardy space that is
defined using isotropic dilations $r (x,u) := (r x, r u)$ but which is also compatible with the
Heisenberg group structure. For similar reasons, this was also the case for M\"uller and Seeger
in their proof of \eqref{sobolev} where they introduced a local isotropic Hardy space $h^1_{iso}(G)$
and prove that $m_{1, d-1} =  e^{ i \sqrt{\mathcal L}} (1 + \sqrt{\mathcal L})^{- (d-1)/2}$ is bounded from
$h^1_{iso}(G)$ to $L^1(G)$. Our main effort will be to establish that for $\theta \not=1$, 
$m_{\theta, d}$ is bounded from $h^1_{iso}(G)$ to $L^1(G)$.

\subsection*{A local, isotropic Hardy space}
In \cite{MSeeger} a local, isotropic Hardy space $h^1_{iso}(G)$ was introduced in their study
of the wave equation on Lie groups $G$ of Heisenberg type. The classical Hardy space $H^1(G)$ is defined
with respect to the homogeneous balls
$$
B_r(x,u) \ = \ \{(y,v) \in G:  \|(y,v)^{-1} \cdot (x,u) \|_K \, \le \, r \ \}
$$
so that $|B_r(x,u)| = c  r^{d_1 + 2 d_2} = c r^Q$. The space
$h^1_{ iso}(G)$ is defined with respect to {\it isotropic balls} skewed by the Heisenberg group translation
$$
B_r^E(x,u) \ := \ \{(y,v) \in G:  \|(y,v)^{-1} \cdot (x,u) \|_E \, \le \, r \ \}
$$
where $\|(x,u)\|_E := |x| + |u|$ is comparable to the classical Euclidean norm on ${\mathbb R}^{d_1} \times
{\mathbb R}^{d_2}$. Hence $|B_r^E(x,u)| = c r^d =  c r^{d_1 + d_2}$.

For $0<r \le 1$ we define a $(P, r)$ atom as a
function $b$ supported in the isotropic Heisenberg ball $B_r^E(P)$ with radius $r$ and
centre $P$, such that $\|b\|_2 \le r^{-d/2}$, and further, in the case where $r\le 1/2$, we require the cancellation 
$\int b = 0$. 
A function $f$ belongs to $h^1_{iso}(G)$ if $f = \sum c_{\nu} b_{\nu}$ 
where $b_{\nu}$ is a $(P_{\nu}, r_{\nu})$ atom for some centre $P_{\nu}$ and radius $r_{\nu}\le 1$
with $\sum |c_{\nu}| < \infty$. 
The norm of $f\in h^1_{iso}(G)$ is defined as
$$
\|f\|_{h^1_{iso}(G)} \ := \ \inf \sum_{\nu} |c_{\nu}|
$$
where the infimum is taken over all representations $f = \sum_{\nu} c_{\nu} b_{\nu}$ where each $b_{\nu}$
is an atom. The space $h^1_{ iso}(G)$ is a closed subspace of $L^1(G)$ and for every $1<p<2$,
$L^p(G)$ is a complex interpolation space for the couple $(h^1_{ iso}(G), L^2(G))$. See \cite{MSeeger}
where these facts are established. 

\section{A more general framework}\label{framework}
For the oscillating spectral multipliers $m_{\theta, \beta}$ in \eqref{osc-multipliers},
large $\lambda \gg 1$ are the important spectral frequencies and so on the
convolution kernel side, small balls are the important ones. Hence the local
space $h^1_{iso}(G)$ is the relevant one for these multipliers.

We now describe a more general framework of spectral multipliers which include both the class
of Mihlin-H\"ormander multipliers as well as the oscillating examples $m_{\theta, \beta}$. 
We introduce a general class of spectral multipliers which satisfy some
scale-invariant conditions (as in \eqref{hormander}) and which depend on an oscillation parameter $\theta \ge 0$,
a decay parameter $\beta \ge 0$ and a smoothness parameter $s\ge 0$. We introduce
the following class $M_{\theta, \beta, s}$ of
spectral multipliers. Fix a nontrivial cut-off $\chi \in C^{\infty}_c(0,\infty)$.  When $ 0 \le t \le 1$, 
we impose the standard uniform $L^2$ Sobolev norm condition 
\begin{equation}\label{hypothesis-lie-neg}
\sup_{0\le t \le 1} \| m(t \cdot) \chi\|_{L^2_s({\mathbb R}_{+})} \ < \ \infty
\end{equation}
and for $1\le t$, we impose the conditions
\begin{equation}\label{hypothesis-lie-pos}
\sup_{1\le t} \,  t^{\theta \beta/2} \|m(t \cdot ) \chi \|_{L^{\infty}({\mathbb R}_{+})}{\color{black}<\infty},  
\  \ {\color{black}\sup_{1\le t} t^{- \theta (2s - \beta) /2} \|m(t \cdot) \chi\|_{L^2_s({\mathbb R}_{+})} \ < \ \infty}. 
\end{equation}
The conditions \eqref{hypothesis-lie-neg} and \eqref{hypothesis-lie-pos} do not depend on the choice of $\chi$.
For $m \in M_{\theta, \beta, s}$, we define $C_m^{\theta,\beta,s}$ to be the maximum of
the quantities appearing in \eqref{hypothesis-lie-pos}.

When $\theta = 0$, these conditions reduce to the condition
$\sup_{t\ge 0} \|m(t \cdot) \chi\|_{L^2_s} < \infty$ and if this holds for some $s > d/2$, the fundamental work
of Hebisch \cite{hebisch} and M\"uller-Stein \cite{MS} {\color{black}show} that the spectral multiplier operator is bounded on all $L^p(G), 1<p<\infty$.

The examples 
$m_{\theta, \beta}(\lambda) =  e^{i \lambda^{\theta}} \lambda^{-\theta \beta/2} \chi(\lambda)$ when $\theta > 0$
from \eqref{osc-multipliers} satisfy conditions \eqref{hypothesis-lie-neg} and
\eqref{hypothesis-lie-pos}. 
Note that \eqref{hypothesis-lie-pos} expresses a growth in the $L^2_s$ Sobolev norm of $m(t \cdot) \chi$
(when $s > \beta/2$) and a decay in the $L^{\infty}$ norm of $m(t \cdot) \chi$.
If the $L^2$ Sobolev condition in \eqref{hypothesis-lie-pos} is satisfied for some $s>0$,  then by complex
interpolation, it is also satisfied for all $0\le s' \le s$ with $C_m^{\theta, \beta, s'} \lesssim C_m^{\theta, \beta, s}$
since the $s' = 0$ case
$\|m( t \cdot ) \chi\|_{L^2} \lesssim t^{- \theta \beta/2}$ is implied by the $L^{\infty}$ condition.

Define
\begin{equation}\label{M-beta-lie}
{\mathcal M}_{\beta}  \ := \ \bigcup_{\theta \ge 0, \theta \not= 1, s>d/2} M_{\theta, \beta, s} .
\end{equation}
This puts us in the position to employ analytic interpolation arguments 
to deduce that $m \in {\mathcal M}_{\beta}$ is an $L^p(G)$ multiplier in the range $ d |1/p - 1/2| \le \beta/2$
from a Hardy space bound for multiplier operators associated to $m\in {\mathcal M}_d$. 
Furthermore, from the invariance of ${\mathcal M}_d$ under multiplication by $\lambda^{iy}$
for any real $y$ (with resulting polynomial in $y$ bounds in \eqref{hypothesis-lie-neg} and \eqref{hypothesis-lie-pos}),
it suffices to show $m(\sqrt{{\mathcal L}}) : h^1_{iso}(G) \to L^1(G)$ for $m\in {\mathcal M}_d$.

The main result here is to establish the endpoint bound, that a strong-type $L^p(G)$ estimate holds for
$m(\sqrt{\mathcal L})$ with $m\in {\mathcal M}_{\beta}$ and $\beta < d$, when $\beta/2 = d |1/p - 1/2|$.

\begin{theorem}\label{endpoint-Q} If $\beta < d$, then for any $m \in {\mathcal M}_{\beta}$,
$m(\sqrt{{L}})$ is 
bounded on all $L^p(G)$ with $\beta/2 \ge d|1/p - 1/2|$.
\end{theorem}
We will use an analytic interpolation argument to establish Theorem \ref{endpoint-Q}. First we
we smoothly decompose the multiplier
\begin{equation}\label{high-low}
m(t) \ = \  m_{small}(t) + m_{large}(t)
\end{equation}
into high and low frequency parts 
where $m_{small}(t) = m(t)$ for small $0<t \lesssim 1$ and 
$m_{large}(t) = m(t)$ for large $t \gg 1$.
We note that
$m_{small}$ is a Mihlin-H\"ormander multiplier and we appeal to established results (see for instance \cite{hebisch}). 

Therefore it suffices to  treat the operator
$m_{large} (\sqrt{\mathcal L})$. The key estimate for these operators is contained 
in the following theorem. 

\begin{theorem}\label{hardy-iso} Suppose that $m \in M_{\theta, d, s}$ for some $\theta > 0, \, \theta\not=1$ and some $s>d/2$.
Let $m_{large}$ be the large frequency part of $m$
as described in \eqref{high-low}. Then 
$$
\|m_{large}(\sqrt{\mathcal L}) f \|_{L^1(G)} \ \lesssim \ C_m^{\theta,d,s} \, \|f\|_{h^1_{{iso}}(G)} 
$$
where we recall $C_m^{\theta, d, s}$ is the maximum of the quantities appearing in \eqref{hypothesis-lie-pos}.
\end{theorem}

To see how to complete the proof of Theorem \ref{endpoint-Q} from Theorem \ref{hardy-iso}, let
$T(\sqrt{\mathcal L}) = m_{large}(\sqrt{\mathcal L})$ be the operator associated to the large frequencies
of an $m$ in the statement of Theorem \ref{endpoint-Q}. 
It suffices
to establish an $L^p(G)$ bound at the endpoint $\beta/2 = d |1/p - 1/2|$ for $T(\sqrt{\mathcal L})$
when
$m \in {\mathcal M}_{\beta}$. In particular $m \in M_{\theta, \beta, s}$ for some $\theta \ge 0$, $\theta \not= 1$
and some $s > d/2$. It suffices to consider the case $\theta>0$ since the case $\theta = 0$ corresponds to Mihlin-H\"ormander
multipliers which have been successfully treated in the setting of groups of Heisenberg type; see  \cite{hebisch} or \cite{MS}. 

For $z \in {\mathbb C}$ such that ${\rm Re}(z) \in [0,1]$, consider
$$
{\mathfrak m}^z(\lambda) \ := \ m_{large}(\lambda) \lambda^{\theta /2 (\beta - dz)} 
$$
and denote by $T_z (\sqrt{\mathcal L})$ the associated spectral multiplier operator. Note that $T_{\sigma} = T$
where $\sigma = \beta/d$ and $1/p = \sigma + (1- \sigma)/2$ when $1< p< 2$ and $\beta/2 = d|1/p - 1/2|$.
If ${\rm Re}(z) = 0$,
then $\|{\mathfrak m}^{iy}\|_{L^{\infty}} \le C_{m}^{\theta,\beta,s}$ since $m \in M_{\theta, \beta, s}$ and hence 
$T_{iy}(\sqrt{\mathcal L})$ is bounded
on $L^2(G)$,  uniformly in $y\in {\mathbb R}$ . When $z = 1 + iy$ we have ${\mathfrak m}^{1+iy} \in  M_{\theta, d, s}$
with 
$$
C_{{\mathfrak m}^{1+iy}}^{\theta,d,s} \ \lesssim \ |y|^s \, C_m^{\theta,\beta,s}
$$
implying that 
$T_{1+iy}(\sqrt{\mathcal L})$ maps $h^1_{iso}(G)$ into $L^1(G)$ with polynomial bounds
in $y$  by Theorem \ref{hardy-iso}.
Hence $T(\sqrt{\mathcal L})$ is bounded on $L^p(G)$ by Stein's analytic interpolation theorem.

In proving Theorem \ref{hardy-iso}, we will not explicitly state the constants $C_m^{\theta,d,s}$ when applying the condition \eqref{hypothesis-lie-pos}, but the dependence will be clear whenever this condition is applied.

We follow an argument developed in \cite{CW} which establishes a Hardy space $H^1(G)$ bound for oscillating multipliers on general stratified groups. Our original hope was to adapt this argument, using only the bounds \eqref{L1} and \eqref{finite-speed} from \cite{MSeeger}.
However, at the final step of the argument, we need to employ the full analysis of the wave operator
as detailed in \cite{MSeeger}.

\section{The first part of the proof of Theorem \ref{hardy-iso}}\label{main-H1-estimate}

 We fix a spectral multiplier $m \in M_{\theta, d, s}$  for some $\theta \not= 1$ with $\theta > 0$
and some $s > d/2$ and consider the large frequency part $m_{large}$ as described in \eqref{high-low}. 

We fix an atom $a_B$ supported in an isotropic, Heisenberg ball $B = B^E_r(P)$ with $r\le 1$ and we want to prove
\begin{equation}\label{L1-atom}
\int_G |T(\sqrt{\mathcal L}) a_B (x) | \, dx \ \lesssim \ 1
\end{equation}
where $T(\sqrt{\mathcal L}) = m_{large}(\sqrt{\mathcal L})$. We decompose $m_{large} = \sum_{j>0} m_j$
where $m_j(t) = m(t) \phi(2^{-j} t)$ for an appropriate smooth $\phi$ supported in $[1/2,2]$. Hence
$T(\sqrt{\mathcal L}) =
\sum_{j> 0} m_j (\sqrt{\mathcal L})$ and only the conditions for $m$ in \eqref{hypothesis-lie-pos} are relevant.

Without loss of generality we may assume that the ball $B$ is centered at the origin, $P = 0$. 
Let $L\le 0$ be such that $2^{L-1} < r \le 2^L$.
The $L^2$ boundedness of $T(\sqrt{\mathcal L})$
implies that, for any fixed $C>0$,
$$
\int_{\|(x,u)\|_E\le C 2^L} |T(\sqrt{\mathcal L}) a_B(x,u)| dx du \ \lesssim \ 1
$$ 
via the Cauchy-Schwarz inequality and so it suffices
to show that 
\begin{equation}\label{desire-2}
\int_{\|(x,u)\|_E \gg 2^L} |T(\sqrt{\mathcal L}) a_B (x,u)| \, dx du \ \lesssim \ 1.
\end{equation}

We bound the integral in \eqref{desire-2} by 
$$
{\mathcal I} \ := \ \int_{\|(x,u)\|_E \gg 2^L} |m_j(\sqrt{\mathcal L}) a_B (x,u)| dx du .
$$
Writing $m^j(\lambda) := m_j(2^j \lambda) = m(2^j \lambda) \phi(\lambda)$,
we have $m_j(\sqrt{\mathcal L}) = m^j (2^{-j} \sqrt{\mathcal L})$. Using
\eqref{fourier-rep}, we write 
$$
m^j(\sqrt{\mathcal L}) \ = \ \int {\widehat{ m}^j}(\tau) \chi( \sqrt{\mathcal L}) 
e^{i  \tau \sqrt{\mathcal L}} \, d\tau
$$
where $\chi \equiv 1$ on the support of $\phi$ so that
\begin{equation}\label{mj-formula}
m_j(\sqrt{\mathcal L}) \ = \ \int {\widehat{ m}^j}(\tau) \chi( 2^{-j} \sqrt{\mathcal L}) 
e^{i  2^{-j} \tau \sqrt{\mathcal L}} \, d\tau.
\end{equation}
We remark that in the $\tau$ integral above, we may assume $|\tau| \gg 1$ since
$$
\int_{|\tau| \lesssim 1}  |{\widehat{m}^j}(\tau)| \|\chi(2^{-j} \sqrt{\mathcal L}) 
e^{i 2^{-j} \tau \sqrt{\mathcal L}}\|_{L^1 \to L^1} \, d\tau  \ \lesssim \
\int_{|\tau| \lesssim 1}  |{\widehat{m}^j}(\tau)| \, d\tau \ \lesssim \ 2^{-j\theta d/2}
$$
by the uniform boundedness of $ \|\chi(2^{-j} \sqrt{\mathcal L}) 
e^{i 2^{-j} \tau \sqrt{\mathcal L}}\|_{L^1 \to L^1}$ for small $|\tau|$ (by dilation-invariance of the
$L^1$ operator norm and an application of Hulanicki's theorem \cite{h}),
and the $L^{\infty}$ assumption on ${m}^j$. This is summable over $j>0$ since $\theta > 0$.

Henceforth we shall assume $\tau$ is large in the integral \eqref{mj-formula} representing
$m_j(\sqrt{\mathcal L})$.

We split the integral ${\mathcal I} = I + II$ into two parts where
\begin{align*}
I + II :=
\sum_{j\in \mathcal{A}} \int_{\|(x,u)\|_E \gg 2^L} |m_j(\sqrt{\mathcal L}) a_B (x,u)| dx du \ \\ + \
\sum_{j \in \mathcal{B}} \int_{\|(x,u)\|_E \gg 2^L} |m_j(\sqrt{\mathcal L}) a_B (x,u)| dx du
\end{align*}

where $\mathcal{A} = \{ j > 0: \  j(1-\theta) + L \le 0 \}$ and 
$\mathcal{B} =  \{ j > 0: \ j(1-\theta) + L > 0 \}$. 
Note that the set $\mathcal{B}$ is empty for $\theta>1$.

For $j\in \mathcal{B}$, 
\begin{align*}
\int_{\|(x,u)\|_E \gg 2^L} &|m_j(\sqrt{\mathcal L}) a_B(x,u)| dx du \  
\\
&\leq\int_G |a_B(y,v)| 
\Bigl[ \int_{\|(x,u)\|_E \gg 2^L} |m_j(\sqrt{\mathcal L})((y,v)^{-1} \cdot (x,u))| dx du \Bigr] dy dv 
\end{align*}
and when $\|(x,u)\|_E \gg 2^L$, we have $\|(y,v)^{-1}\cdot (x,u)\|_E \ge 2^L$ whenever $\|(y,v)\|_E \le 2^L$.
This follows easily from the group law
$$
(y,v)^{-1}\cdot (x,u) \ = \ (x - y, u - v + 1/2 \langle J x, y \rangle ),
$$
noting that $L\le 0$.
Hence 
$$
\int_{\|(x,u)\|_E \gg 2^L} |m_j(\sqrt{\mathcal L}) a_B(x,u)| dx du \  \le \ 
\int_{\|(x,u)\|_E \geq 2^L} |m_j(\sqrt{\mathcal L})(x,u)| dx du, 
$$
where we used the fact that for atoms  $\|a_B\|_{L^1(G)}\lesssim 1$.

Let us denote by $K_{\tau}(x,u)$ the convolution kernel of 
$\chi(\tau^{-1} \sqrt{\mathcal L}) e^{i \sqrt{\mathcal L}}$. Then for $g = 2^j/\tau$,
$(K_{\tau}(x,u))_g := g^Q K_{\tau}(\delta_g (x,u))$
is the convolution kernel associated to $\chi(2^{-j} \sqrt{\mathcal L}) e^{i 2^{-j} \tau \sqrt{\mathcal L}}$. Hence
by \eqref{mj-formula},
$$
\int_{\|(x,u)\|_E \ge 2^L} |m_j(\sqrt{\mathcal L})(x,u)| dx du \ \le \ 
\int  |{\widehat{m}^j}(\tau)| \Bigl[ \int_{\|(x,u)\|_E \ge 2^L} |(K_{\tau})_g(x,u)| dx du \Bigr] d\tau .
$$
But $\|\delta_{g^{-1}} (x,u) \|_E \ge 2^L$ implies $\|\delta_{\tau}(x,u)\|_E \ge 2^{j+L}$ since $j > 0$.
Hence we have
$$
\int_{\|(x,u)\|_E \ge 2^L} |(K_{\tau})_g(x,u)| dx du \ \le \ \int_{\|\delta_{\tau}(x,u)\|_E \ge 2^{j+L}}
|K_{\tau}(x,u)| \, dx du.
$$

Recall the pointwise estimate \eqref{finite-speed} for $K_{\tau}(x,u)$ (valid for large $\tau$);
$$
|K_{\tau}(x,u)| \ \lesssim_N  \ |\tau|^M \, \|\delta_{\tau} (x,u) \|_E^{-N}
$$
for some $M\ge 1$ and any $N\ge 1$. Hence
if $|\tau| \le 2^{j+L}$, choosing $N$ large enough we have
$$
\int_{\|\delta_{\tau}(x,u)\|_E \ge 2^{j+L}} |K_{\tau}(x,u)| \, dx du \ \lesssim_N \ 2^{-N(j+L)}
$$
and therefore
\begin{align*}
\int_{|\tau| \le 2^{j+L}}  
|{\widehat{m}^j}(\tau)| \Bigl[ \int_{\|(x,u)\|_E \ge 2^L} |(K_{\tau})_g(x,u)| dx du \Bigr] d\tau\\
\ \lesssim_N \ 2^{-N(j+L)} \int_{|\tau| \le 2^{j+L}} |{\widehat{ m}^j}(\tau)| \, d\tau
\end{align*}
which by Cauchy-Schwarz and \eqref{hypothesis-lie-pos} is at most 
$$
2^{-(N-1/2)(j+L)} \| {\widehat{m}^j} \|_{L^2} \ \lesssim \ 2^{-(N-1/2)(j+L)} 2^{-j\theta d/2}
\ \le \ 2^{-(N-1/2) (j + L)}.
$$
Note that $j+L>0$ for $j\in\mathcal{B}$, when $\mathcal{B}$ is non empty. Indeed, in this case we must have $0<\theta<1$, which implies $j+L>j(1-\theta)+L>0$.

To treat the remaining part
$$
\int_{|\tau| \ge 2^{j+L}}  
|{\widehat{m}^j}(\tau)| \Bigl[ \int_{\|(x,u)\|_E \ge 2^L} |(K_{\tau})_g(x,u)| dx du \Bigr] d\tau,
$$
we promote the integration over $\|(x,u)\|_E \ge 2^L$ to all of $G$ so that the inner integral is at most
\begin{align*}
\int_{|\tau| \ge 2^{j+L}}  
&|{\widehat{m}^j}(\tau)| \Bigl[ \int_{G} |(K_{\tau})_g(x,u)| dx du \Bigr] d\tau \\&=
\int_{|\tau| \ge 2^{j+L}}  
|{\widehat{m}^j}(\tau)| \Bigl[ \int_{G} |K_{\tau}(x,u)| dx du \Bigr] d\tau
\lesssim \ \int_{|\tau| \ge 2^{j+L}}  |{\widehat{m}^j}(\tau)| |\tau|^{(d-1)/2} \
\\ &\le \
2^{-(\sigma -1/2)(j+L)} \|{{m}^j}\|_{L^2_{\sigma + (d-1)/2}} \ \lesssim \
2^{-(\sigma -1/2)(j(1-\theta) + L)}
\end{align*}
for some $\sigma > 1/2$. 
The first inequality uses \eqref{L1} in Theorem \ref{MS-basic} followed by the Cauchy-Schwarz inequality and
our hypothesis \eqref{hypothesis-lie-pos} on the $L^2$ Sobolev norms of ${m}^j$. Therefore 
$$
\int_{\|(x,u)\|_E \ge 2^L} |m_j(\sqrt{\mathcal L})(x,u)| \, dx du \ \lesssim \ 
2^{-(\sigma - 1/2)(j(1-\theta) + L)}
$$ 
and so
$$
\int_{\|(x,u)\|_E \gg 2^L} |m_j(\sqrt{\mathcal L}) a_B(x,u)| \, dx du \ \lesssim \ 
2^{-(\sigma - 1/2)(j(1-\theta) + L)},
$$
showing that we can sum over all $j\in \mathcal{B}$ (since $\theta \not= 1$) and uniformly bound $II$.   

For $I$, we split $\mathcal{A} = \mathcal{A}_{-} \cup \mathcal{A}_{+}$ further such that $\mathcal{A}_{-} = \{ j \in \mathcal{A}: j + L \le 0\}$ and
$\mathcal{A}_{+} = \{j \in \mathcal{A}: j + L > 0\}$. This splits $I = I_{-} + I_{+}$ accordingly.

For the sum over $j \in \mathcal{A}_{+}$, we split each integral in $I_{+}$ into two parts
$$
 \int_{\|(x,u)\|_E \gg 2^L} |m_j(\sqrt{\mathcal L})  a_B (x,u)| dx du \ = \ S_{\Lambda, j} + L_{\Lambda, j}
 $$
 for some positive $\Lambda > 0$ where
 $$
 S_{\Lambda, j} \ := \ \int_{2^{L} \ll \|(x,u)\|_E \le 2^{L+\Lambda}} |m_j(\sqrt{\mathcal L})  a_B (x,u)| dx du 
$$
and $L_{\Lambda, j}$ is defined similarly but with the integration taken over $(x,u)$ such that 
$2^{L+\Lambda} \le \|(x,u)\|_E$.

For $S_{\Lambda, j}$ we use Cauchy-Schwarz, the $L^2$ bound $\|a_B\|_2 \le 2^{-Ld/2}$ on our atom
and our $L^{\infty}$ hypothesis on $m_j$  to see that
$$
S_{\Lambda, j} \ \le \ 2^{(L + \Lambda)d/2} \| m_j(\sqrt{\mathcal L}) a_B \|_{L^2} \ \lesssim \
2^{(L + \Lambda)d/2} 2^{-j\theta d/2} \|a_B\|_{L^2} \ \le \ 2^{\Lambda d/2} 2^{-j\theta d/2}.
$$
Next we will show that
\begin{equation}\label{L-Lambda}
L_{\Lambda, j} \ \lesssim_{\sigma} \ 2^{(\sigma - 1/2) j \theta} 2^{- (\sigma - 1/2) (j + L + \Lambda)}
\end{equation}
for some $\sigma > 1/2$. Choosing $\Lambda$ such that 
$$
2^{\Lambda} \ = \ 2^{j\theta} 2^{- \frac{\sigma -1/2}{\sigma - 1/2 + d/2} (j+ L)}
$$
optimizes the bound for the sum $S_{\Lambda, j} + L_{\Lambda, j}$. 

With this choice of $\Lambda$ (which is $>0$ since $j\in \mathcal{A}$) we have 
$S_{\Lambda, j} + L_{\Lambda, j} \lesssim  2^{-\epsilon (j + L)}$ where 
$$
\epsilon \ = \ \frac{d}{2} \cdot \frac{\sigma -1/2}{\sigma - 1/2 + d/2} \ > \ 0
$$
and this shows that 
$$
I_{+} \ = \ \sum_{j\in A_{+}} \bigl[ S_{\Lambda, j} + L_{\Lambda, j} \bigr]  \ \lesssim \ 1
$$
is uniformly bounded. We now turn to establish \eqref{L-Lambda}.

Proceeding as for terms $j \in \mathcal{B}$, and using the formula \eqref{mj-formula}, we see that
$$
L_{\Lambda, j} \ \le \ \int |{\widehat{m}^j}(\tau)| \Bigl[ \int_{2^{L+\Lambda} \le \|(x,u)\|_E} |(K_{\tau})_g(x,u)|
 \, dx du \Bigr] \, d\tau \ =: \ L_{\Lambda, j}'
$$
where we recall that $g = 2^j/\tau$, 
$(K_{\tau})_g(x,u) = g^Q K_{\tau}(\delta_g (x,u))$ and $K_{\tau}(x,u)$ is the convolution kernel of 
$\chi(\tau^{-1} \sqrt{\mathcal L}) e^{i \sqrt{\mathcal L}}$. 
We split $L_{\Lambda, j}' = III + IV$ where
$$
III \ := \ 
 \int_{|\tau| \le 2^{j+L+\Lambda}} 
 |{\widehat{m}^j}(\tau)| \Bigl[ \int_{2^{L+\Lambda} \le \|(x,u)\|_E} |(K_{\tau})_g(x,u)| \, dx du \Bigr] \, d\tau
 $$
 and $IV$ is defined in the same way except the integration in $\tau$ is over $|\tau| \ge 2^{j+L+\Lambda}$.
 
Again since
$\|\delta_{g^{-1}} (x,u) \|_E \ge 2^{L+\Lambda}$ implies 
$\|\delta_{\tau}(x,u)\|_E \ge 2^{j+L+\Lambda}$ since $j  > 0$,
we have
$$
\int_{\|(x,u)\|_E \ge 2^{L+\Lambda}} |(K_{\tau})_g(x,u)| dx du \ \le \ \int_{\|\delta_{\tau}(x,u)\|_E \ge 2^{j+L+\Lambda}}
|K_{\tau}(x,u)| \, dx du
$$
and so 
$$
III \ \le \ 
 \int_{|\tau| \le 2^{j+L+\Lambda}} 
 |{\widehat{m}^j}(\tau)| \Bigl[ \int_{2^{j+L+\Lambda} \le \|\delta_{\tau}(x,u)\|_E} |K_{\tau}(x,u)| \, dx du \Bigr] \, d\tau
 $$
 which by \eqref{finite-speed} (recall we have reduced to large $\tau$) implies that
 $$
 III \ \lesssim_N \ 2^{-N(j+L+\Lambda)} \int_{|\tau| \le 2^{j+L+\Lambda}} |{\widehat{m}^j}(\tau)| \, d\tau
 \ \lesssim \ 2^{-(N-1/2)(j+L+\Lambda)} 2^{-j\theta d/2}
 $$
 for any $N\ge 1$.
For IV, since $\tau$ is large we can use \eqref{L1},
Cauchy-Schwarz and our $L^2$ Sobolev condition on ${m}^j$ to see that
$$
IV \ \lesssim \int_{2^{j+L+\Lambda} \le |\tau|}  |{\widehat{m}^j}(\tau)| |\tau|^{(d-1)/2} \, d\tau \ \lesssim_{\sigma} \
2^{-(\sigma -1/2)(j+L+\Lambda)} 2^{(\sigma - 1/2) j \theta}
$$
for some $\sigma > 1/2$. Hence
$L_{\Lambda, j} \le III + IV \lesssim 2^{(\sigma - 1/2) j \theta} 2^{-(\sigma-1/2)(j + L + \Lambda)}$ holds
for some $\sigma > 1/2$, establishing \eqref{L-Lambda}.

Finally we turn to the sum over $j\in \mathcal{A}_{-}$ where $j+L \le 0$ and it is here that we would like to use
the cancellation of the atom $a_B$. We are allowed to use the cancellation since if there are $j$'s such that $j+L\leq0$, we must have $L\leq-1$, so that the atom is supported in a ball of radius $r\leq1/2$. We have
$$
m_j (\sqrt{\mathcal L})  a_B (x,u) = \int \bigl[ m_j(\sqrt{\mathcal L})((y,v)^{-1} \cdot (x,u)) - m_j(\sqrt{\mathcal L})(x,u)\bigr]
a_B(y,v) dy dv,
$$
and through basic estimates we can bound the difference by $\|\delta_{2^j}(y,v)\|_K = 2^j \|(y,v)\|_K$. If this Kor\'anyi norm could be controlled by
the Euclidean norm, then since $\|(y,v)\|_E \le 2^L$ for $(y,v) \in B$, we would gain a factor of $2^{j+L}$ which
is summable for $j\in \mathcal{A}_{-}$.  It is here where we see the two incompatible geometries (coming from Kor\'anyi norm on the
one hand, and the Euclidean norm on the other) creating an obstacle. 

To overcome this, we will need to employ 
a refined decomposition in \cite{MSeeger} of the
operator $\chi(\tau^{-1} \sqrt{\mathcal L}) e^{i\sqrt{\mathcal L}}$ appearing in the formula
\eqref{mj-formula} 
$$
m_j(\sqrt{\mathcal L}) \ = \ \int {\widehat{m}^j}(\tau) \chi(2^{-j} \sqrt{\mathcal L}) 
e^{i 2^{-j} \tau \sqrt{\mathcal  L}} \, d\tau .
$$
In the next section we will describe this decomposition but now we make a couple preliminary remarks. First
recall that we may assume $\tau$ is large in the above integral.

Second we can easily handle any error terms $E_{\tau,j}({\mathcal L})$ arising in our decomposition of
$\chi(2^{-j} \sqrt{\mathcal L}) e^{i 2^{-j} \tau \sqrt{\mathcal L}}$ which have a uniform $L^1$ operator norm of $O(|\tau|^{-2})$
since
\begin{equation}\label{error-tau-2}
\int_{|\tau|\gg 1}  |{\widehat{m}^j}(\tau)|
\|E_{\tau, j}{\mathcal L}\|_{L^1 \to L^1} \, d\tau  \ \lesssim \
\int_{|\tau| \gg 1}  |{\widehat{m}^j}(\tau)| \, |\tau|^{-2} d\tau \ \lesssim \ 2^{-j\theta d/2}
\end{equation}
which again is summable for $j> 0$.

\section{A refined decomposition of $\chi(\tau^{-1} \sqrt{\mathcal L}) e^{i\sqrt{\mathcal L}}$}\label{decomposition}

In this section we will recall a decomposition of the operator
$\chi(\tau^{-1} \sqrt{\mathcal L}) e^{i\sqrt{\mathcal L}}$ from \cite{MSeeger} which we will need
to carry out the final step in the argument. For the reader's convenience, we give an outline of how this decomposition is derived
but refer to \cite{MSeeger} for the precise details of each step.

The decomposition is based on the following subordination
formula which relates the wave operator $e^{i \sqrt{\mathcal L}}$ to the Schr\"odinger group $\{e^{i t {\mathcal L}}\}$.
See Proposition 4.1 in \cite{MSeeger}.

\begin{lemma}\label{subordination} Let $\chi_1 \in C^{\infty}_c$ be equal to 1 on the support of $\chi$. Then
there exists smooth functions $a_{\tau}(s)$ and $\rho_{\tau}(s)$ supported for $s \sim 1$ such that
\begin{equation}\label{subordination-formula}
\chi(\tau^{-1} \sqrt{\mathcal L}) e^{i \sqrt{\mathcal L}} \ = \ \chi_1(\tau^{-2} {\mathcal L}) \sqrt{\tau}
\int e^{i \frac{\tau}{4s}} a_{\tau}(s) e^{i s {\mathcal L}/ \tau} ds \ + \ \rho_{\tau}(\tau^{-2} {\mathcal L}).
\end{equation}
Furthermore the $L^p(G)$ operator norm of $\rho_{\tau}(\tau^{-2} {\mathcal L})$ is $O(|\tau|^{-N})$
for every $N\ge 1$.
\end{lemma}

Hence by \eqref{error-tau-2}, the contribution corresponding to the term $\rho_{\tau}(\tau^{-2} {\mathcal L})$
can be easily treated, being uniformly summable in $j>0$. 

The formula \eqref{subordination-formula} is an immediate consequence of the spectral resolution \eqref{spectral-resolution}
for $\sqrt{\mathcal L}$ and the stationary phase formula
\begin{equation}\label{stationary-phase}
\chi(\sqrt{x}) e^{i \lambda\sqrt{x}} \ = \ \chi_1(x) \sqrt{\lambda}
\int e^{i\frac{\lambda}{4s}} a_{\lambda}(s) e^{i \lambda s x} ds \ + \ \rho_{\lambda}(x),
\end{equation}
together with the change of variables $x \to x/\lambda^2$. The phase $\Phi(s) = \lambda x s + \lambda/4s$
in the oscillatory integral in \eqref{stationary-phase}
has a unique nondegenerate critical point at $s_{*} = 1/2\sqrt{x}$ which is comparable to $1$ due to the support of $\chi$
and $\chi_1$. Also $\Phi(s_{*}) = \lambda \sqrt{x}$ and $\Phi''(s_{*}) = 4 \lambda x^{3/2}$. Hence the stationary
phase formula says
$$
\chi_1(x) \sqrt{\lambda} \int e^{i \Phi(s)} a_{\lambda}(s) \, ds \ = \ c \chi_1(x) 
\frac{ e^{i \Phi(s_{*})} }{\sqrt{\Phi''(s_{*})}}+ {\rm error} \ = \ {\tilde \chi}(\sqrt{x}) e^{i \lambda \sqrt{x}} + {\rm error}
$$
but of course a careful analysis of this principle with uniform control of the various terms, especially the error term,
is required to make this precise and useful. See \cite{MSeeger} for details.

Lemma \ref{subordination} allows us to understand the wave operator $e^{i\sqrt{\mathcal L}}$ via the
Schr\"odinger group $\{ e^{i t {\mathcal L}}\}$ where explicit formulae are well known. For instance
the convolution kernel $S_t$ for $e^{i t {\mathcal L}}$ is given by
\begin{equation}\label{Kt}
S_t(x,u) \ = \ \int_{{\mathbb R}^{d_2}} \Bigl( \frac{|\mu|}{2\sin(2\pi t |\mu|)}\Bigr)^{d_1/2} 
e^{-i |x|^2 \frac{\pi}{2} |\mu| \cot(2\pi t |\mu|)} e^{2\pi i u \cdot \mu} d\mu.
\end{equation}
From \eqref{subordination-formula}, we see that
$\chi(\tau^{-1} \sqrt{\mathcal L}) e^{i \sqrt{\mathcal L}}  =   \chi_1(\tau^{-2} {\mathcal L}) n_{\tau}(\mathcal L)
+ \rho_{\tau}(\tau^{-2} {\mathcal{L}})$
where
$$
n_{\tau}(\mathcal L) \ = \ 
\sqrt{\tau} \int_{\mathbb R}  e^{i\frac{\tau}{4 s}} a_{\tau}(s) e^{i s {\mathcal L}/\tau} \, ds
$$
and so by \eqref{Kt},
the convolution kernel of $n_{\tau}(\mathcal L)$ is given by the formula 
\begin{align*}
n_{\tau}(x,u) =
 \sqrt{\tau}
\int_{{\mathbb R}^{d_2}} \int_{\mathbb R}  e^{i \frac{\tau}{4s}} a_{\tau}(s) 
\Bigl( \frac{|\mu|}{2\sin(2\pi s |\mu|/\tau)}\Bigr)^{d_1/2} \\ \times 
e^{-i |x|^2 \frac{\pi}{2} |\mu| \cot(2\pi s |\mu|/\tau)} e^{2\pi i u \cdot \mu} ds d\mu.
\end{align*}
As the phase above exhibits periodic singularities (where $\sin(2\pi s |\mu|/\tau) = 0$), 
it is natural to introduce an equally spaced decomposition
in the central Fourier variables $|\mu|$: that is, with respect to the spectrum of the operator $|U|$ (recall that
$U := ( - i U_1, \ldots, - i U_{d_2})$). 

We fix an $\eta_0 \in C^{\infty}_c({\mathbb R})$, supported in a small neighborhood of the origin and such
that $\sum_{k \in {\mathbb Z}} \eta_0(t - k \pi) \equiv 1$. We decompose
$n_{\tau}(\mathcal L) = n_{\tau}^0({\mathcal L}, |U|) + n_{\tau}^1({\mathcal L}, |U|)$ where 
$$
 n_{\tau}^0({\mathcal L}, |U|) \ = \ 
\sqrt{\tau} \int_{\mathbb R}  e^{i\frac{\tau}{4 s}} a_{\tau}(s) \eta_0( \frac{s}{\tau} |U|) e^{i s {\mathcal L}/\tau} \, ds
$$
and $n_{\tau}^1({\mathcal L}, |U|) = \sum_{k\ge 1} n_{\tau, k}({\mathcal L}, |U|)$ where
$$
n_{\tau,k}({\mathcal L}, |U|) \ = \ 
\sqrt{\tau} \int_{\mathbb R}  e^{i\frac{\tau}{4 s}} a_{\tau}(s) \eta_0( \frac{s}{\tau} |U| - k \pi) e^{i s {\mathcal L}/\tau} \, ds.
$$
The refined $L^1$ estimates established in \cite{MSeeger} which prove \eqref{L1} in Theorem \ref{MS-basic} are the following.

\begin{proposition}\label{n0-1} We have
\begin{equation}\label{0}
\|n_{\tau}^0({\mathcal L}, |U|)\|_{L^1 \to L^1} \ \lesssim \ (1+ |\tau|)^{(d - 1)/2}
\end{equation}
and for each $k\ge 1$,
\begin{equation}\label{k}
\|n_{\tau, k}({\mathcal L}, |U|)\|_{L^1 \to L^1} \ \lesssim \ k^{-(d_1 + 1)/2} \, (1+ |\tau|)^{(d - 1)/2}.
\end{equation}
\end{proposition}
Since $d_1 \ge 2$, the sum over $k\ge 1$ converges and so Theorem \ref{MS-basic} is an immediate consequence
of Proposition \ref{n0-1}.

We denote by ${\mathcal V}_{\tau}$ the operator
$\chi_1(\tau^{-2} {\mathcal L}) n_{\tau}^0({\mathcal L}, |U|)$
and by ${\mathcal W}_{\tau}$ the operator \\
$\chi_1(\tau^{-2} {\mathcal L}) n_{\tau}^1({\mathcal L}, |U|)$ so that
\begin{equation}\label{V-W}
\chi(\tau^{-1} \sqrt{\mathcal L}) e^{i \sqrt{\mathcal L}} \ = \ {\mathcal V}_{\tau} \ + \ {\mathcal W}_{\tau} \ + \
E_{\tau}
\end{equation}
where the operator $E_{\tau}$ is negligible; the $L^1$ operator norm is $O(|\tau|^{-N})$ for any $N\ge 1$ 
and so by \eqref{error-tau-2}
gives rise to a bound of the form $2^{-j\theta d/2}$. 

Finally we need to decompose ${\mathcal W}_{\tau} = \sum_{n \ge 0} {\mathcal W}_{\tau, n}$ which will
localize $|U|/\tau$ (or the central Fourier variables $|u|/\tau$).  Fix two smooth
$\zeta_0, \zeta_1$ with $\zeta_0$ supported in $(-1,1)$ and $\zeta_1$ supported in $\pm (1/2, 2)$ so that
$\zeta_0(t) + \sum\limits_{n\ge 1} \zeta_1( 2^{-n} t) \equiv 1$. Define
$$
{\mathcal W}_{\tau, 0} := \zeta_0(\tau^{-1} |U|) {\mathcal W}_{\tau} \ \ {\rm and \ for} \ n\ge1, \ \ 
{\mathcal W}_{\tau, n} := \zeta_1(\tau^{-1} 2^{-n} |U|) {\mathcal W}_{\tau}
$$
so that ${\mathcal W}_{\tau} = \sum_{n\ge 0} {\mathcal W}_{\tau, n}$ holds. This finer decomposition
will allow us to identify when we will be able to implement the argument outlined at the end of previous section
which uses the cancellation of the atom $a_B$.

\section{The final step in the proof of Theorem \ref{hardy-iso}}\label{final-step}

Recall that the last step in the proof of Theorem \ref{hardy-iso} is to uniformly bound
\begin{equation}\label{last-step}
I_{-} \ := \ \sum_{j>0: j+L \le 0} \int_{\|(x,u)\|_E \gg 2^L} |m_j(\sqrt{\mathcal L}) a_B (x,u)| dx du \ =: \
\sum_{j+L\le 0} I_j .
\end{equation}

From the previous section, we see that the
operator $\chi(2^{-j}\sqrt{\mathcal L}) e^{i 2^{-j} \tau \sqrt{\mathcal L}}$ appearing in the formula
\eqref{mj-formula} 
$$
m_j(\sqrt{\mathcal L}) \ = \ \int {\widehat{m}^j}(\tau) \chi(2^{-j} \sqrt{\mathcal L}) 
e^{i 2^{-j} \tau \sqrt{\mathcal  L}} \, d\tau 
$$
can be written as $({\mathcal V}_{\tau})_g + ({\mathcal W}_{\tau})_g + (E_{\tau})_g$ where
the $L^1$ operator norm of $(E_{\tau})_g$  is uniformly $O(|\tau|^{-N})$ for any $N\ge 1$, 
$g = 2^j/\tau$ and the convolution kernel of $({\mathcal W}_{\tau})_g$, say,  is given by the $L^1$ invariant dilate
$({\mathcal W}_{\tau})_g (x,u) = g^Q {\mathcal W}_{\tau} (\delta_g (x,u))$.

From \eqref{error-tau-2}, we have the bound $I_j \le I_j^1 + I_j^2 + O(2^{-j\theta d/2})$
where
\begin{equation}\label{Ij1}
I_j^1 =  \int_{\|(x,u)\|_E \gg 2^L}  \Bigl| \int_{|\tau| \gg 1} {\widehat m}^j(\tau) ({\mathcal V}_{\tau})_g a_B(x,u) d\tau \Bigr|
dx du 
\end{equation}
and
$$
I_j^2 =  \int_{\|(x,u)\|_E \gg 2^L}  \Bigl| \int_{|\tau| \gg 1} {\widehat m}^j(\tau) ({\mathcal W}_{\tau})_g  a_B(x,u) d\tau \Bigr|
dx du .
$$
We first treat $I_j^2$, noting
$$
I_j^2 \le  \int_{|\tau|\gg 1}  |{\widehat{m}^j}(\tau)| \| ({\mathcal W}_{\tau})_g a_B \|_{L^1(G)} \, d\tau
\le  \sum_{n\ge 0} \int_{|\tau|\gg 1}  |{\widehat{m}^j}(\tau)| \| ({\mathcal W}_{\tau,n})_g a_B \|_{L^1(G)} \, d\tau
$$
where we employed the decomposition ${\mathcal W}_{\tau} = \sum_{n\ge 0} {\mathcal W}_{\tau,n}$ from the
end of the previous section. 

From \eqref{k} in Proposition \ref{n0-1}, we see that $\|n_{\tau}^1\|_{L^1 \to L^1} \lesssim (1+|\tau|)^{(d-1)/2}$
and so $\|{\mathcal W}_{\tau,0}\|_{L^1 \to L^1} \lesssim (1+|\tau|)^{(d-1)/2}$. Also, we can write
$$
{\mathcal W}_{\tau, n} = \zeta_1(\tau^{-1} 2^{-n} |U|) {\mathcal W}_{\tau} \ = \ 
\sum_{k \ge n + C} \zeta_1(\tau^{-1} 2^{-n} |U|) n_{\tau, k},
$$
where $C$ depends only on support properties of $\eta_0$ and $a_\tau$, thus the estimate $\|{\mathcal W}_{\tau, n} \|_{L^1 \to L^1} \lesssim 2^{-n (d_1 - 1)/2} (1+|\tau|)^{(d-1)/2}$
follows again from \eqref{k}. Hence 
\begin{equation}\label{W-L1}
 \|{\mathcal W}_{\tau, n} \|_{L^1 \to L^1} \ \lesssim \ 2^{-n (d_1 - 1)/2} (1+|\tau|)^{(d-1)/2}
\end{equation}
holds for all $n\ge 0$. 

We now split the sum over $n \ge 0$; for large $n$, we will use the exponential decay in $n$ in the
estimate \eqref{W-L1} and here we will not need the cancellation of the atom $a_B$. For small $n$,
we will use the cancellation of the atom but the incompatibility of the two geometries will diminish
the favorable factor $2^{j+L}$ coming from the difference with an exponential
growth in $n$ which nonetheless will be admissible if $n$ is small enough.

Recall that when $j\in \mathcal{A}_{-}$, we have $2^{j+L}\le 1$ and we split the sum in $n$, writing $I_j^2\le I_{j,1}^2 + I_{j,2}^2$ where
$$
I_{j,1}^2 \  := \ \sum_{n\in {\mathcal N}_1}  \int_{|\tau|\gg 1}  |{\widehat{m}^j}(\tau)|
\|({\mathcal W}_{\tau,n})_g a_B \|_{L^1(G)} \, d\tau
$$
and ${\mathcal N}_1 = \{n\ge 0 : 2^{-\epsilon(j+L)} \le 2^n \}$ for some fixed small $\epsilon > 0$. 
We define $I_{j,2}^2$ similarly
where now the sum over $n$ lies in ${\mathcal N}_2 = \{ n\ge 0: 2^n \le 2^{-\epsilon (j+L)} \}$. 

For $I_{j,1}^2$ we use the bound in \eqref{W-L1} to conclude that 
$$
I_{j,1}^2 \ \lesssim \ \sum_{n\in {\mathcal N}_1} 2^{- n (d_1 -1)/2} \int  |{\widehat{m}^j}(\tau)| |\tau|^{(d-1)/2} \, d\tau \ \lesssim \
2^{\delta (j+L)} \int |{\widehat{m}^j}(\tau)| |\tau|^{(d-1)/2} \, d\tau
$$
for some $\delta > 0$. But
$$
 \int |{\widehat{m}^j}(\tau)| |\tau|^{(d-1)/2} \, d\tau \ = \ \int_{|\tau| \le 2^{j \theta}} 
 |{\widehat{m}^j}(\tau)| |\tau|^{(d-1)/2} \, d\tau + \int_{|\tau| \ge 2^{j \theta}} 
 |{\widehat{m}^j}(\tau)| |\tau|^{(d-1)/2} \, d\tau 
$$
and the first integral is uniformly bounded in $j$ by an application of Cauchy-Schwarz and our $L^{\infty}$ condition on 
${m}^j$.
The second integral is also uniformly bounded in $j$ by another application of Cauchy-Schwarz and our $L^2$ Sobolev condition
on ${m}^j$. Hence $I_{j,1}^2 \lesssim 2^{\delta (j+L)}$ which is uniformly summable over $j$ with $j + L \le 0$. 

To treat $I_{j,2}^2$, we make the understanding that $\zeta = \zeta_0$ when $n=0$ and $\zeta = \zeta_1$ when
$n\ge 1$. Hence
$$
({\mathcal W}_{\tau, n})_g ({\mathcal L}) = n_{\tau}^1(\tau^2 2^{-2j} {\mathcal L},\tau 2^{-j} |U|) 
\chi_1(2^{-2j} {\mathcal L}) \zeta(2^{-j-n} |U|)
$$ 
and so $({\mathcal W}_{\tau, n})_g ({\mathcal L}) a_B (x,u) = ({\mathcal W}_{\tau, n})_g ({\mathcal L}) (a_B * H_{j,n})(x,u)$ where $H_{j,n}$ is defined as
$$
f * H_{j,n}  \ := \ \chi_1'(2^{-2j}{\mathcal L}) \zeta'(2^{-j-n} |U|) f,
$$
where $\chi_1'$ and $\zeta'$ are smooth cut-off functions which are identically equal to 1 on the
supports of $\chi_1$ and $\zeta$, respectively. Hence
$\chi_1 = \chi_1 \chi_1'$ and $\zeta = \zeta \zeta'$.

Therefore
$$
I_{j,2}^2 \ \le \ \sum_{n\in {\mathcal N}_2}  \, 
\int_{|\tau|\gg 1}  |{\widehat{m}^j}(\tau)| \|{\mathcal W}_{\tau,n}\|_{L^1 \to L^1}
\|a_B * H_{j,n}\|_{L^1(G)} \, d\tau.
$$
We note that $H_{j,n}$ is a $\delta_{2^j}$ dilate of the convolution
of the Schwartz function $h_1$ (coming from $\chi_1^{\prime}(\mathcal L)$) and the $2^n$ dilate
of the convolution kernel of $\zeta'(|U|)$ which has the form $\delta \otimes h_2$ where
$\delta$ is the Dirac measure on ${\mathbb R}^{d_1}$ and $h_2$ is a Schwartz function on ${\mathbb R}^{d_2}$.
Hence
\begin{equation}\label{Hjn}
H_{j,n}(x,u) = \int 2^{j(d_1 + 2d_2)} h_1(2^j x, 2^{2j} w) 2^{d_1(j+n)} h_2(2^j(u - w)) \, dw  
\end{equation}
and in particular, we see that $\|H_{j,n}\|_{L^1} \lesssim 1$.

We now
use the cancellation of the atom $a_B$. We have
\begin{align*}
a_B &* H_{j,n} (x,u) \ = \ \int a_B(y,v) \bigl[ H_{j,n}((y,v)^{-1}\cdot (x,u)) - H_{j,n}(x,u) \bigr] dy dv
\\
&= - \int a_B(y,v) \Bigl( \int_0^1 \bigl\langle y, \nabla_x H_{j,n}(x - s y, u - sv + \frac{1}{2} \langle J x, y \bigr\rangle ) \rangle \ \ \ \ \ \ \ \ 
\\
& +  \bigl\langle v + \frac{1}{2} \langle J x, y \rangle, \nabla_u H_{j,n}(x - s y, u - sv + \frac{1}{2} 
\langle J x, y \rangle ) \bigr\rangle
\, ds \Bigr) dy dv.
\end{align*}
Using $\langle Jx, y\rangle = \langle J(x-sy), y\rangle$, we see that
$$
\|a_B * H_{j,n}\|_{L^1(G)} \lesssim 2^L \bigl[ \|\nabla_x H_{j,n}\|_{L^1(G)} + \|\nabla_u H_{j,n}\|_{L^1(G)} +
\| |x| \nabla_u H_{j,n}\|_{L^1(G)} \bigr]
$$
and so \eqref{Hjn} implies $\|a_B * H_{j,n}\|_{L^1(G)} \lesssim 2^n 2^{j+L}$.

We should mention that the bound $\|a_B * H_{j,n}\|_{L^1} \lesssim \min(1, 2^n 2^{j+L})$
was used in \cite{MSeeger} (see Lemma 9.1 there). We have reproduced a sketch the proof here for the
convenience of the reader. 

The bound $\|a_B * H_{j,n}\|_{L^1(G)} \lesssim 2^n 2^{j+L}$, together with 
$\|{\mathcal W}_{\tau,n}\|_{L^1 \to L^1} \lesssim  |\tau|^{(d-1)/2}$ shows
$$
I_{j,2}^2 \ \lesssim \ \sum_{n\in {\mathcal N}_2} 2^n 2^{j+L} \int  |{\widehat{m}^j}(\tau)| |\tau|^{(d-1)/2} \, d\tau \ \lesssim \
2^{\delta (j+L)} \int |{\widehat{m}^j}(\tau)| |\tau|^{(d-1)/2} \, d\tau
$$
for some $\delta > 0$. Since we have seen that the above integral is uniformly bounded in $j$, we see that
\begin{equation}\label{Ij2}
|I_j^2| \ \lesssim \ 2^{\delta(j+L)} \ \ {\rm for \ some} \ \delta > 0
\end{equation}
which is uniformly summable over $j$ with $j + L \le 0$.
This gives a proof that $\sum_{j: j+L\le 0} I_j^2 \lesssim 1$.

It remains to treat the terms $I_j^1$ in \eqref{Ij1}. Recall that 
${\mathcal V}_{\tau} = \chi_1(\tau^{-2} {\mathcal L}) n_{\tau}^0({\mathcal L}, |U|)$ and the
convolution kernel of $ n_{\tau}^0({\mathcal L}, |U|)$ is given by the formula
\begin{align*}
K_{\tau}^0(x,u)= \sqrt{\tau}
\int_{{\mathbb R}^{d_2}} \int_{\mathbb R}  e^{i \frac{\tau}{4s}} a_{\tau}(s) 
\eta_0(s |\mu|/ \tau)
\Bigl( \frac{|\mu|}{2\sin(2\pi s |\mu|/\tau)}\Bigr)^{d_1/2} \\ \times
e^{-i |x|^2 \frac{\pi}{2} |\mu| \cot(2\pi s |\mu|/\tau)} e^{2\pi i u \cdot \mu} ds d\mu.
\end{align*}
A straightforward analysis of the phase in the above oscillatory integral representation of the
kernel $K_{\tau}^0$ shows a couple of regions in $(x,u)$ space where the phase is nondegenerate 
and an integration by parts argument shows the following rapid decay estimates (see \cite{MSeeger}, Lemma 6.2
for details).

\begin{proposition}\label{K0} For every $N\ge 1$, we have
$$
|K_{\tau}^0(x,u)| \ \lesssim_N \ (1+|\tau|)^{(Q+1)/2 - N} (|x|^2 + |u|)^{-N}, \ \ {\rm when} \ |x|^2 + |u| \ge 2
$$
and
$$
|K_{\tau}^0(x,u)| \ \lesssim_N \ (1+|\tau|)^{(Q+1)/2 - N} (1 + |u|)^{-N}, \ \ {\rm when} \ |x|^2 \le  1/20.
$$
\end{proposition}
 
Now let $\phi \in C^{\infty}_c({\mathbb R}^{d_1 + d_2})$ be such that $\phi(x,u) = 1$ when $|x|^2 + |u| \le 2$
and $\phi(x,u) = 0$ when $|x|^2 + |u| \ge 3$. Also fix $\psi \in C^{\infty}_c({\mathbb R}^{d_1})$ such that
$\psi(x) = 1$ when $|x|^2 \le 1/40$ and $\psi(x) = 0$ when $|x|^2 \ge 1/20$. Decompose 
$K_{\tau}^0 = K_{\tau}^{0,1} + K_{\tau}^{0,2}$ where
$$
K_{\tau}^{0,1}(x,u) \ := \ K_{\tau}^0(x,u) \phi(x,u) (1 - \psi(x))  
$$
so that by Proposition \ref{K0}, $\|K_{\tau}^{0,2}\|_{L^1} \lesssim |\tau|^{-N}$ for any $N\ge 1$. Furthermore
from Hulanicki's result in \cite{h},
$$
\|{\mathcal V}_{\tau} - n_{\tau}^0\|_{L^1 \to L^1} \ = \ \|\chi_1(\tau^{-2} {\mathcal L}) n_{\tau}^0 - n_{\tau}^0\|_{L^1\to L^1}
\ \lesssim \ (1+|\tau|)^{-N}
$$
for every $N\ge 1$ (see Lemma 6.1 in \cite{MSeeger}). Hence if ${\mathcal K}_{\tau}^{0,1}$ denotes
the operator of convolution with $K_{\tau}^{0,1}$, then $\|{\mathcal V}_{\tau} - {\mathcal K}_{\tau}^{0,1}\|_{L^1 \to L^1}
\lesssim (1+|\tau|)^{-N}$ and so if
$$
II_j^1 \ := \ \int_{\|(x,u)\|_E \gg 2^L}  \Bigl| \int_{|\tau| \gg 1} {\widehat m}^j(\tau) a_B * (K_{\tau}^{0,1})_g (x,u) d\tau \Bigr|
dx du,
$$
then
\begin{equation}\label{Ij-IIj}
I_j^1 \lesssim  II_j^1 \ +  \ \int_{|\tau|\gg 1} |{\widehat m}^j(\tau)| \| ({\mathcal V}_{\tau} - {\mathcal K}_{\tau}^{0,1})_g
 \|_{L^1 \to L^1} d\tau \ = \ II_j^1 +
O(2^{-j\theta d/2}).
\end{equation}
For a fixed $\tau$ and $(x,u)$ let us examine the convolution 
\begin{align*}
a_B * (K_{\tau}^{0,1})_g (x,u) = \int  g^Q K_{\tau}^0( \delta_g ((y,v)^{-1}\cdot (x,u))) \phi(\delta_g ((y,v)^{-1}\cdot (x,u)))\\ \times(1 - \psi(2^j/\tau(x -y)) \,
a_B(y,v) dy dv.
\end{align*}
Note that $|x-y| \sim |\tau|/2^j$ but since $2^{j+L} \le 1 \ll |\tau|$, we have $|y| \le 2^L \ll |\tau|/2^j$ and so
$|x| \sim |\tau|/2^j$ or $|\tau| \sim 2^j |x|$. Therefore
$$
II_j^1 \ = \ \int_{\|(x,u)\|_E \gg 2^L}  \Bigl| \int\limits_{|\tau| \gg 1, \ |\tau| \sim 2^j |x|}
         {\widehat m}^j(\tau) a_B * (K_{\tau}^{0,1})_g (x,u) d\tau \Bigr| dx du.
$$

Now we work backwards, recalling that the operator decomposition $n_{\tau}^0 = {\mathcal K}_{\tau}^{0,1} + 
{\mathcal K}_0^{0,2}$ so that 
$$ 
{\mathcal V}_{\tau} = {\mathcal K}_{\tau}^{0,1} +  [{\mathcal V}_{\tau} - n_{\tau}^0] + {\mathcal K}_{\tau}^{0,2}.
$$
Furthermore \[\chi(\tau^{-1} \sqrt{\mathcal L}) e^{i \sqrt{\mathcal L}} =
 {\mathcal V}_{\tau} \ + \ {\mathcal W}_{\tau} \ + \
E_{\tau} =  {\mathcal K}_{\tau}^{0,1} \ + \   {\mathcal W}_{\tau} \ + \
 [{\mathcal V}_{\tau} - n_{\tau}^0] + {\mathcal K}_{\tau}^{0,2} \ + \
E_{\tau}.
\]
The terms $ [{\mathcal V}_{\tau} - n_{\tau}^0], {\mathcal K}_{\tau}^{0,2}$ and 
$E_{\tau}$ are negligible, each with an $L^1$ operator norm of  $O(|\tau|^{-2})$ which gives rise to the admissible 
bound $2^{-j\theta d/2}$ by \eqref{error-tau-2}. Hence,
using the estimates in \eqref{Ij2}, we see that
$$
|II_j^1| \ \lesssim \  \int_{\|(x,u)\|_E \gg 2^L}  \Bigl| \int_{|\tau| \sim 2^j |x|} 
{\widehat m}^j(\tau) a_B * (K_{\tau})_g (x,u) d\tau \Bigr| dx du \ + \ 2^{\delta(j+L)} \ + \ 2^{-j\theta d/2}
$$
where $K_{\tau}$ denotes (as before) the convolution kernel of $\chi(\tau^{-1} \sqrt{\mathcal L}) e^{i \sqrt{\mathcal L}}$.

Therefore, setting
$$
III_j^1 \ := \ \int_{\|(x,u)\|_E \gg 2^L}  \Bigl| \int\limits_{|\tau|\gg 1, |\tau| \sim 2^j |x|} 
{\widehat m}^j(\tau) a_B * (K_{\tau})_g (x,u) d\tau \Bigr| dx du,
$$
matters are reduced to showing $\sum_{j+L\le 0} III_j^1 \lesssim 1$. We have $III_j^1 \le $
$$
\int_{|\tau|\gg 1} |{\widehat m}^j(\tau)| \int |a_B(y,v)| \Bigl[ \int_{|x| \sim |\tau|/2^j}
g^Q | K_{\tau}(\delta_g ( (y,v)^{-1} \cdot (x,u))) | dx du \Bigr] dy dv \, d\tau.
$$ 
Recall that $(y,v)^{-1}\cdot (x,u) = (x - y, u - v + (1/2) \langle Jx, y \rangle)$. In the $(x,u)$ integral above,
make the change of variables $u' = u-v + (1/2) \langle J x, y \rangle$, followed by $x' = x -y$, noting that
$$
|x-y| \sim |\tau|/2^j \ \Leftrightarrow \ |x| \sim |\tau|/2^j \ \ {\rm since} \ \ |y| \le 2^L \ll |\tau|/2^j,
$$
to conclude
\begin{align*}
III_j^1 \ \le \ \int_{|\tau|\gg 1} |{\widehat m}^j(\tau)|   \int_{|x| \sim |\tau|/2^j}
g^Q | K_{\tau}(\delta_g (x,u)) | dx du  \,  d\tau
\\
\le \  \int_{|\tau|\gg 1} |{\widehat m}^j(\tau)|   \int_{|x| \sim1}
| K_{\tau}(x,u) | dx du  \,  d\tau.
\end{align*}
By \eqref{finite-speed}, we have the pointwise estimate $|K_{\tau}(x,u)| \lesssim |\tau|^M (|\tau x| + |\tau^2 u|)^{-N}$
for some $M\ge 1$ and every $N\ge 1$ whenever $\tau$ is large. Therefore
$$
 \int_{|x| \sim1} |K_{\tau}(x,u)| dx du \ \lesssim \ \frac{1}{|\tau|^{N-M}} \int_{{\mathbb R}^{d_2}} \frac{1}{(1+ |\tau u|)^N} du
\ = \ c |\tau|^{-(N+d_2 - M)}
$$
and so
$$
III_j^1 \ \lesssim \ \int_{|\tau|\gg 1} |{\widehat m}^j(\tau)| |\tau|^{-2} \, d\tau \ \lesssim \ 2^{-j\theta d/2}
$$
which is summable for $j>0$ with $j+L\le 0$. 

The completes the proof of Theorem \ref{hardy-iso}.



\end{document}